\documentclass[12pt]{amsart}

\usepackage{amsfonts,amsmath,latexsym,amssymb,verbatim,amsbsy,amsthm}

\theoremstyle{plain}
\newtheorem{THEOREM}{Theorem}[section]
\newtheorem{theorem}[THEOREM]{Theorem}

\newtheorem{lemma}[THEOREM]{Lemma}

\theoremstyle{definition}

\theoremstyle{remark}
\newtheorem{remark}[THEOREM]{Remark}
\newtheorem{claim}[THEOREM]{Claim}

\newcommand{\thm}[1]{Theorem~\ref{#1}}
\newcommand{\lem}[1]{Lemma~\ref{#1}}

\def \a {\alpha}
\def \b {\beta}
\def \g {\gamma}
\def \d {\delta}

\def \g {\gamma}
\def \e {\varepsilon}
\def \f {\varphi}
\def \l {\lambda}
\def \n {\nabla}
\def \s {\sigma}
\def \th {\theta}
\def \o {\omega}
\def \r {\rho}

\def \D {\Delta}

\def \cE {\mathcal{E}}

\newcommand{\R}{\ensuremath{\mathbb{R}}}   

\def \p {\partial}
\def \ra {\rightarrow}

\def \loc {\mathrm{loc}}

\newcommand{\der}[2]{#1 \cdot \nabla #2}
\newcommand{\ave}[2]{\langle #1 \rangle_{#2}}
\newcommand{\aave}[3]{\langle #1 \rangle_{#2,#3}}

\DeclareMathOperator{\diver}{div} %
\DeclareMathOperator{\dom}{Dom} %
\DeclareMathOperator{\tr}{Tr} %
\DeclareMathOperator{\vol}{Vol} %

\begin{document}

\title[Self-similar blowup for the Euler equation]
{On the energy behavior of locally self-similar blowup for the Euler equation}
\author{Anne Bronzi}
\thanks{The work of A. Bronzi is supported by CNPq, Conselho Nacional de Desenvolvimento Cient\'ifico e Tecnol\'ogico - Brasil, grant 236994/2012-3}
\author{Roman Shvydkoy}
\thanks{The work of R. Shvydkoy is partially supported by NSF grant DMS--1210896}
\address[A. Bronzi and R. Shvydkoy]
{Department of Mathematics, Stat. and Comp. Sci.\\
M/C 249,\\
      University of Illinois\\
      Chicago, IL 60607} 
\email{annebronzi@gmail.com}
\email{shvydkoy@uic.edu}


\maketitle

\begin{abstract}

In this note we study locally self-similar blow up for the Euler equation. The main result states  that under a mild $L^p$-growth assumption on the profile $v$, namely, $\int_{|y| \sim L} |v|^p dy \lesssim L^{\g}$ for some $\g <p-2$, the self-similar solution carries a positive amount of energy up to the time of blow-up $T$, namely, $\int_{|y| \sim L} |v|^2 dy \sim L^{N-2\a}$. The result implies and extends several previously known exclusion criteria. It also supports a general conjecture relating fractal local dimensions of the energy measure with the rate of velocity growth at the time of possible blowup.

\end{abstract}


\section{Description of the result}
Let $u\in C([0,T), H^s(\mathbb R^N))$, for some $s > \frac{N}{2} + 1$, $N\geq 3$, be a solution of
the Euler equations:
\begin{equation}\label{EE}
\begin{split}
u_t + \der{u}{u} + \n p & =0\\
\n \cdot u & = 0.
\end{split}
\end{equation}
The pressure can be recovered from the Poisson equation $\D p = -\diver(\diver(u\otimes u))$.
Up to a harmonic polynomial the solution is given by
\begin{equation}\label{pressure}
p(x)=-\frac{|u(x)|^2}{N}+P.V.\int_{\R^N}K_{ij}(x-y)u_i(y)u_j(y)dy,
\end{equation}
where $K_{ij}(y) = \frac{y_i y_j - \frac{\d_{ij}}{N}|y|^2}{ \o_N |y|^{N+2}}$, 
and $\o_N = 2\pi^{N/2}(N\Gamma(N/2))^{-1}$ is the volume of the unit ball in $\R^N$. Since, in view of \eqref{EE}, $\n p$ has to decay at infinity, formula \eqref{pressure} defines the only solution up to a constant. In this note we study locally self-similar solutions given by 
\begin{equation}\label{self-similar}
u(x,t)=\frac{1}{(T-t)^{\frac{\alpha}{1+\alpha}}}v\left(\frac{x-x_0}{
(T-t)^{\frac{1}{1+\alpha}}}\right)
\end{equation}
in a ball $x\in B_{\r_0}(x_0)$, $t<T$, and some fixed $\alpha>0$, and we assume that the profile field is locally smooth, $v\in C^3_{\loc}(\R^N)$. Our motivation to study such a blow-up scenario comes from abundant numerical evidence suggesting that singular solutions tend to form self-similar structures, in anisotropic fashion  \cite{kerr}, isotropic vortex knot formations \cite{pelz}, and more recently on the boundary of a fluid domain \cite{guo} to mention a few. Rigorous analysis of \eqref{self-similar} has a relatively recent history (see works of Chae \cite{cha,chae-11,chae-shv}, He \cite{he-bdd,he-ext} and Schonbeck \cite{schon}). Below we will recite results that are most relevant to this present note. For now let us address two important issues that arise directly from the set up. First, in much of the literature the ansatz \eqref{self-similar} is postulated along with the corresponding pressure 
\begin{equation}\label{e:pres1}
p (x,t) = \frac{1}{(T-t)^{\frac{2\alpha}{1+\alpha}}}q\left(\frac{x-x_0}{
(T-t)^{\frac{1}{1+\alpha}}}\right),
\end{equation}
in the same ball $B_{\r_0}(x_0)$. While in globally self-similar case ($\rho_0 = \infty$) this form of the pressure can be easily justified, in the local case it becomes overdetermined as $p$ is already recovered via \eqref{pressure}. To resolve the problem we show in \lem{l:pressure} that indeed \eqref{e:pres1} holds up to a time dependent constant $c(t)$ with at most polynomial growth as $t \ra T$, and the pair $(v,q)$ solves \eqref{EE} in self-similar variables.

Second, the conservation of total energy $\|u(t)\|_2 = \|u_0\|_2$ implies in particular that the energy in the ball $B_{\r_0}(x_0)$ remains bounded. Since 
\begin{equation}\label{en-ball}
\|u(t)\|_{L^2(B_{\r_0}(x_0))}^2  = \frac{1}{L^{N-2\a}}\int_{|y|<\r_0 L} |v(y)|^2 dy,
\end{equation}
for $L = (T-t)^{-1/(1+\a)}$, this implies the bound
\begin{equation}\label{h1}
\int_{|y|<L} |v(y)|^2 dy \lesssim L^{N-2\a}.
\end{equation}
Here and in the future, $A\lesssim B$ means $A/B$ is bounded for large $L$, and $A\sim B$ means $A \lesssim B$ and $B \lesssim A$. Our main result states that under a mild growth bound on higher $L^p$-norms one can reverse inequality \eqref{h1}.

\begin{theorem}\label{t:main} Suppose $u\in C([0,T), H^s(\mathbb R^N))$ is a solution to \eqref{EE} locally self-similar in a ball $B_{\r_0}(x_0)$ with profile $v\in C^3_{\loc}(\R^N)$ and scaling $0<\a<\frac{N}{2}$. Suppose further that for some $p\geq 3$ and $\g <p-2$,
\begin{equation}\label{e:p-bound}
\int_{|y|\sim L} |v(y)|^p dy \lesssim L^{\g}, \text{ for large } L.
\end{equation}
Then either $v = 0$ or one has
\begin{equation}\label{e:2-bound}
L^{N-2\a} \lesssim \int_{|y|<L} |v(y)|^2 dy \lesssim L^{N-2\a}.
\end{equation}
\end{theorem}
Let us note again that the upper bound in \eqref{e:2-bound} is simply a consequence of the fact that $v$ is a part of the solution $u$ with finite energy. 
Before we embark on the proof, let us discuss applications of \thm{t:main} and its relation to previously known results.

\begin{remark}[Energy concentration] In view of \eqref{en-ball}, the conclusion of the theorem states that unless the profile $v$ is trivial, the self-similar blowup carries some positive amount of energy with it, i.e. $\|u(t)\|_{L^2(B_{\r_0}(x_0))}$ stays bounded away from zero as time $t$ approaches critical. The energy behavior at the time of blow-up can be described in more details in terms of the energy measure introduced in \cite{shv-en}. Energy measure is simply the weak$^*$-limit of $|u(x,t)|^2 dx$, as $t \ra T$, denoted $\cE_T$. As a consequence of \eqref{e:2-bound},  
\[
\cE_T(B_\r(x_0)) \sim \lim_{L \ra \infty} \frac{1}{L^{N-2\a}}\int_{|y|<\r L} |v(y)|^2 dy \sim \r^{N-2\a},
\]
for all small $\r$. This implies that in the case of non-trivial self-similar blow-up satisfying \eqref{e:p-bound} the exact fractal local dimension of $\cE_T$ exists at $x_0$ and is equal to $D = N-2\a$ (see \cite{triebel}). Results on the energy concentration obtained in \cite{shv-en} support the conjecture that $\dim_{\loc}(\cE_T,x_0) = D$ if and only if $\int_{t}^{T} \|u(s)\|_{L^\infty(\mathrm{near}\ x_0)} ds \sim (T-t)^{\frac{2}{N- D+2}}$. For self-similar solutions we have  $\|u(s)\|_{L^\infty(B_{\sim (T-s)^{1/(1+\a)}}(x_0))} \sim (T-s)^{-\a/(1+\a)}$, thus the integral yields the rate of decay of $(T-t)^{1/(1+\a)}$, which indeed coincides with $(T-t)^{\frac{2}{N- D+2}}$ for $D = N - 2\a$, as conjectured.
\end{remark}

\begin{remark}[Exclusion results] \thm{t:main} can serve as an exclusion result in those cases when \eqref{e:2-bound} is not valid a priori. For example, if in addition $\g < N-p\a$, or if $\g = N-p\a$ and $\int_{|y|\sim L} |v(y)|^p dy \lesssim o(1)$ (in which case necessarily $N<p\a + p-2$), then by the H\"older one also has $\int_{|y|\sim L} |v(y)|^2 dy \lesssim L^{N-2\a} o(1)$. Consequently, since $\a < N/2$, $\int_{|y|<L} |v(y)|^2 dy \lesssim L^{N-2\a}o(1)$, invalidating the lower bound in \eqref{e:2-bound}. This implies $v=0$. The conditions described above hold, in particular, under the assumptions $v \in L^p$ and $\a \leq N/p$, which recovers the exclusion result of Chae and Shvydkoy obtained in \cite{chae-shv}. Moreover, we can see that in the range $N/p<\a < N/2$, \thm{t:main} provides an extension of this result by requesting an extra decay of the $L^p$-norms over the shells $\{|y| \sim L\}$. 

Let us recall another exclusion condition exhibited in \cite{shv-en}:  if $|v(y)| \lesssim |y|^{1-\d}$, for some $\d>0$, and $\int_{|y|<L} |v(y)|^2 dy \lesssim L^{N-2\a}o(1)$ with $\a >\frac{N-2}{4}$, then $v=0$. We can now remove the extra assumption $\a >\frac{N-2}{4}$. Indeed, observe that for all $p>3$ we have 
\[
\int_{|y|\sim L} |v(y)|^p dy \lesssim L^{(p-2)(1-\d)}\int_{|y|\sim L} |v(y)|^2 dy \lesssim L^{(p-2)(1-\d)+N-2\a}.
\]
So, for $p$ large enough, we have $\g = (p-2)(1-\d)+N-2\a < p-2$, and hence $v=0$.
\end{remark}

\begin{remark}[Asymptotic behavior] Let us notice that the following $\a$-point vortex $v(y) = \frac{y^\perp}{|y|^{\a+1}}$ is a stationary solution to the 2D Euler equation, and is also globally self-similar with scaling exponent $\a$ (although not locally smooth). This and results of \cite{he-ext} in exterior domains suggest that asymptotic behavior at $\infty$ should be that of $|y|^{-\a}$ in general. \thm{t:main} expresses this very fact only phrased in terms of $L^2$-averages: $\frac{1}{\vol} \int_{|y|<L}|v|^2 dy \sim L^{-2\a}$. 
\end{remark}
\begin{remark}[Case $\a = N/2$] Let us comment on the energy conservative case $\a = N/2$, not covered by \thm{t:main}. The energy bound \eqref{h1} necessarily enforces the condition $v \in L^2$, which trivially implies the energy drain $\int_{|y|\sim L} |v|^2 dy = o(1)$ in contradiction to \eqref{e:2-bound}. However in this case we can't deduce \eqref{e:2-bound} from any $L^p$-bound on $v$. Instead, it was shown in \cite{shv-en}, via a general result on energy drain, that the decay rate of energy over the shells improves to $\int_{|y|\sim L} |v|^2 dy \lesssim \frac{1}{L^{N+2 - \d'}}$ for any $\d'>0$, provided the sublinear growth bound $|v(y)| \lesssim |y|^{1-\d}$ holds for some $\d>0$. Consequently, $v \in \cap_{\frac{N}{N+1} \leq p \leq N+4} L^p(\R^N)$ (see \cite{shv-en}). We cannot improve upon this result using present technique, however a direct argument can be made via the use of Muckenhoupt weights. 
\end{remark}

Let us finally note that the sublinear growth assumption, $|v(y)| \lesssim |y|^{1-\d}$, is natural in the sense that it breaks the scaling symmetry of the equation in self-similar variables (see \eqref{sseq} below):  if $(v,q)$ is a solution to \eqref{sseq}, then the new pair
\[
v_\l(y) = \l v(y/\l), \quad q_\l(y) = \l^2 q(y/\l)
\]
solves the same equation for any $\l \neq 0$. So, linear solutions to \eqref{sseq} are self-similar in their own sense. One can exhibit many examples  of such solutions:
\begin{equation}\label{e:natur}
\begin{split}
v(y) & = My, \quad q(y) = -\frac{1}{2} \langle (M+M^2)y, y\rangle, \\
\tr M & = 0, \quad M + M^2 \in \mathrm{Sym}_N.
\end{split}
\end{equation}
However, these cannot be a part of locally self-similar blow-up as they violate the energy bound \eqref{h1}.

\section{Recovery of pressure}

For reasons outlined in the introduction, we first have to recover the pressure in self-similar form, and obtain necessary estimates on the profile. This will in fact be the main technical part of the proof of the main theorem. It will be convenient to express various growth bounds in terms of averages over balls
$\ave{f}{L} = \frac{1}{\vol} \int_{|y|<L} f(y) dy$ or over shells $\aave{f}{L_1}{L_2} = \frac{1}{\vol} \int_{L_1<|y|<L_2} f(y) dy$. 

\begin{lemma}\label{l:pressure}
Suppose $(u,p)$ is a solution to \eqref{EE} with $u$ being locally self-similar in the ball $B_{\r_0}(x_0)$ with profile $v\in C^3_{\loc}(\R^N)$, $0<\a < N/2$. Assume that for some $p>2$, $M>0$, 
\begin{equation}\label{temp-v}
\aave{|v|^p}{L}{2L} \lesssim L^M, \text{ for large } L.
\end{equation}
Then the scalar function $q\in C^2_{\loc}(\R^N)$ given by 
\begin{equation}\label{pressure-q}
q(y)=-\frac{|v(y)|^2}{N}+\int_{\R^N}K_{ij}(y-z)v_i(z)v_j(z)dz
\end{equation}
solves the equation
\begin{equation}\label{sseq}
\frac{\a}{1+\a} v + \frac{1}{1+\a} y \cdot \n v + \der{v}{v} + \n q=0,
\end{equation}
and satisfies the bound
\begin{equation}\label{press-bound}
\aave{|q|^{r}}{L}{2L}^{1/r} \lesssim \ave{|v|^2}{L} + \aave{|v|^{2r}}{L/2}{4L}^{1/r}+\sum_{k=1}^\infty \aave{|v|^2}{2^kL}{2^{k+1}L},
\end{equation}
for all  $r>1$. Moreover, there is a bounded function $d(t)$ such that
\[
p(x,t) = \frac{1}{ (T-t)^{\frac{2\a}{\a+1}}} q\left( \frac{x-x_0}{ (T-t)^{\frac{1}{\a+1}} } \right)+ \frac{d(t)}{(T-t)^{\frac{2\a}{\a+1}}}
\]
holds in the ball $|x-x_0| < \r_0$ for all $t$ near $T$.
\end{lemma}

\begin{proof} The strategy of the proof is the following. We first investigate the right hand side of \eqref{pressure-q}, denoted by $I$, and show that it defines a tempered distribution solving the Poisson equation 
\begin{equation}\label{Laplace}
\D I = -\diver(\diver(v \otimes v)).
\end{equation}
We also show that there exists another tempered distribution $q$ solving \eqref{sseq} and hence solving the same equation \eqref{Laplace}. We then conclude that the difference $q-I$ is a harmonic tempered distribution, hence is a polynomial. As a consequence of bounds established on the growth of both $q$ and $I$ at infinity we conclude that $q-I$ is a constant . The bound \eqref{press-bound} will be established in the course of the proof.

 \emph{Step 1}. First, it is easy to see that the integral in \eqref{pressure-q} converges pointwise a.e. Indeed, let us fix $L>1$, and consider a cutoff function $\f_0$ (infinitely smooth, equal $1$ for $|y| <1$, and $0$ for $|y| >2$), and rescaled one $\f_L(y) = \f_0(y/L)$. We obtain
\[
\begin{split}
\int K_{ij}(y-z)v_i(z)v_j(z)dz &= \int K_{ij}(y-z)\f_{3L}(z)v_i(z)v_j(z)dz \\
&+ \int K_{ij}(y-z)(1-\f_{3L}(z))v_i(z)v_j(z)dz\\
&= I_1(y) + I_2(y).
\end{split}
\]
Since $v\in C^3_{\loc}$ one has $I_1 \in C^{\b}$, for all $\b<3$ by the classical Besov estimates. Now, on the ball $|y|<L$, we have 
\[
\begin{split}
|\p^s I_2(y)| &\lesssim \sum_{k=1}^\infty \int_{|z| \sim 2^k L} \frac{1}{|z|^{N+s}} |v(z)|^2 dz \lesssim \sum_{k=1}^\infty  \frac{1}{(2^k L)^s}\ave{|v|^2}{2^k L} \\
& \lesssim  \sum_{k=1}^\infty (2^k L)^{-2\a-s} \lesssim L^{-2\a-s},
 \end{split}
\]
for all $s\geq 0$. This shows that the function defined by
\begin{equation}\label{e:I}
I(y) = -\frac{|v(y)|^2}{N}+\int_{\R^N}K_{ij}(y-z)v_i(z)v_j(z)dz
\end{equation}
is locally smooth and by the classical formula is a solution of the Poisson equation \eqref{Laplace}.

\noindent \emph{Step 2.} Let us now prove the bound \eqref{press-bound} for $I$. This will only be used in this proof to ensure that $I$ is a tempered distribution, when applied to $r = p/2$ and using \eqref{temp-v}, but in the sequel we will make a complete use of it. Let us write $I = -\frac{|v(y)|^2}{N} + J$, where $J$ stands for the integral in \eqref{e:I}. Clearly, only the bound for $J$ is necessary. So,  let us fix an $L>1$, consider $y$ in the shell $\{L<|y|<2L\}$ and $J$ into three integrals (the integrands are suppressed for brevity):
\[
J_1(y) = \int_{|z| <L/2}; \quad J_2(y) = \int_{L/2 < |z|<4L}; \quad J_3(y) = \int_{|z| >4 L}.
\]
Then,
\[
\int_{L< |y| < 2L } |J_1(y)|^{r} dy \leq \int_{L < |y| < 2L } \left( \frac{1}{|y|^{N}}\int_{|z| <L/2} |v(z)|^2 dz \right)^{r} dy \lesssim L^{N}\ave{|v|^2}{L}^r.
\]
Next, by the Calderon-Zygmund boundedness,
\[
\int_{L< |y| < 2L } |J_2(y)|^{r} dy \lesssim \int_{L/2 < |z| < 4L} |v|^{2r} dz \leq L^N  \aave{|v|^{2r}}{L/2}{4L}.
\]
Finally,
\[
\begin{split}
\int_{L< |y| < 2L } |J_3(y)|^{r} dy & =\int_{L< |y| < 2L } \left| \sum_{k=2}^\infty \int_{2^k L <|z|<2^{k+1} L} K_{ij}(y-z) v_i(z)v_j(z) dz \right|^{r} dy \\
&\leq \int_{L< |y| < 2L } \left( \sum_{k=2}^\infty \frac{1}{(2^kL)^N} \int_{2^k L <|z|<2^{k+1} L} |v|^2 dz  \right)^{r} dy \\
&\lesssim L^N \left( \sum_{k=2}^\infty \aave{|v|^2}{2^kL}{2^{k+1}L} \right)^{r}.
\end{split}
\]
This establishes \eqref{press-bound} for $I$. It implies that $I$ is a tempered distribution in view of \eqref{temp-v}, and as previously observed $I$ solves
\eqref{Laplace}.

\noindent \emph{Step 3.} Let us now find a tempered pressure solving \eqref{sseq}. Let us assume for simplicity that $x_0 = 0$. Plugging the ansatz \eqref{self-similar} into \eqref{EE} we see that the expression $(T-t)^{\frac{2\a+1}{\a+1}} \n p(x (T-t)^{\frac{1}{\a+1}}, t)$ is independent of time as long as $|x|\leq \r_0$. Letting $\bar{p}(x (T-t)^{-\frac{1}{1+\a}},t) = p(x,t)$ we conclude that  
$(T-t)^{\frac{2\a}{\a+1}} \n_y \bar{p}(y, t)$ is time independent on the region $|y| \leq \r_0 (T-t)^{-\frac{1}{1+\a}}$. So, in the family of functions 
\[
\left\{(T-t)^{\frac{2\a}{\a+1}} \bar{p}(\cdot, t), \dom = (|y| \leq \r_0 (T-t)^{-\frac{1}{1+\a}})\right\}_{T-t_0<t<T}
\]
the members differ pairwise by constants on their common domains.   Let us pick a monotone sequence $t_n \ra T$, and consider $q_0(y) = (T-t_0)^{\frac{2\a}{\a+1}} \bar{p}(y,t_0)$ defined on $|y| \leq \r_0 (T-t_0)^{-\frac{1}{1+\a}}$. Then for every $n\geq 1$ there exists $c_n\in \R$ such that $q_n(y) + c_n= (T-t_n)^{\frac{2\a}{\a+1}} \bar{p}(y,t_n) $ coincides with $q_0$ on its domain, and therefore $q_n = q_k$ for all $n,k\geq 1$ on the common domain of the pair. This unambiguously defines the function $q(y) = q_n(y)$ for all $|y| \leq \r_0 (T-t_n)^{-\frac{1}{1+\a}}$. For all other values of $t$ we have a scalar function $c(t)$ such that $q(y) + c(t)= (T-t)^{\frac{2\a}{\a+1}} \bar{p}(y,t) $ holds on the ball
$|y| \leq \r_0 (T-t)^{-\frac{1}{1+\a}}$. Thus,
\[
p(x,t) = \frac{1}{ (T-t)^{\frac{2\a}{\a+1}}} q\left( \frac{x}{ (T-t)^{\frac{1}{\a+1}} } \right)+c(t),
\]
for all $|x| \leq \r_0$ and $T-t_0<t<T$. If we plug this back into \eqref{EE} we recover \eqref{sseq} on the whole space. Let us now show that $q$ is a tempered distribution. We have a uniform bound $\|p(t)\|_{1,\mathrm{weak}} \lesssim \|u(t)\|_2 \leq C$. So, $| \{ x: |p(x,t)| >\l\}| \leq \frac{C}{\l}$, for all $t$. Hence, there is $\d>0$ small so that $| \{ x: |p(x,t)| >\frac{1}{\d (T-t)^{\frac{N}{1+\a}}} \}| \leq \frac{\o_N}{2} (T-t)^{\frac{N}{1+\a}}$, where $\o_N$ is the volume of the unit ball in $\R^N$. This implies that in the ball $|x| \leq (T-t)^{\frac{1}{1+\a}}$ there exists a point $x_t$ such that $|p(x_t,t)| \leq \frac{1}{\d (T-t)^{\frac{N}{1+\a}}} $. That implies that there exists $|y_t| \leq 1$ so that $|c(t)| \leq  \frac{1}{ (T-t)^{\frac{2\a}{\a+1}}}|q(y_t)| + \frac{1}{\d (T-t)^{\frac{N}{1+\a}}}$. Since $q$ is locally smooth in the unit ball, we obtain some polynomial bound $|c(t)| \lesssim (T-t)^{-M}$.  On the other hand, denoting $d(t) = -c(t)(T-t)^{\frac{2\a}{1+\a}}$ we have
\begin{equation}\label{qdp}
q(y) = d(t) + (T-t)^{\frac{2\a}{1+\a}}p(y (T-t)^{\frac{1}{1+\a}},t),
\end{equation}
for all $|y| \leq \r_0 (T-t)^{-\frac{1}{1+\a}}$. From the bounds above we thus obtain the following rough bound (here all $M_i \in \R$):
\[
\begin{split}
& \int_{\frac{\r_0}{4 (T-t)^{\frac{1}{1+\a}}} < |y|< \frac{\r_0}{2 (T-t)^{\frac{1}{1+\a}}}  } |q(y)|^{p/2} dy  \lesssim \\
& \lesssim (T-t)^{M_1} + (T-t)^{\frac{p\a-N}{1+\a}} \int_{\r_0/4 <|x|<\r_0/2} |p(x,t)|^{p/2} dx \\
&\lesssim (T-t)^{M_1} + (T-t)^{M_2}( \| u\|^p_{L^p( \r_0/8 <|y|<\r_0)} + \|u\|_{2}^{p/2}) \\
&\lesssim (T-t)^{M_3}. 
\end{split}
\]
Thus, the $L^{p/2}$-integrals of $q$ over dyadic shells grow at most polynomially. This shows that $q$ is a tempered distribution.

\noindent \emph{Step 4.} We now show that $q$ and $I$ differ by a constant. Since they both solve the Laplace equation \eqref{Laplace} and are both distributions on $\R^N$, their difference $q -I$ is a harmonic polynomial $h$. Let us show that $h$ is constant. For all $|y| \leq \frac{\r_0}{2(T-t)^{\frac{1}{1+\a}}}$, we have from \eqref{pressure},
\[
\begin{split}
&(T-t)^{\frac{2\a}{1+\a}} p(y (T-t)^{\frac{1}{1+\a}}, t) =  \\ &= -\frac{1}{N}|v(y)|^2 + \int_{|z| \leq \r_0} K_{ij}( y (T-t)^{\frac{1}{1+\a}} - z) (v_iv_j)( z/(T-t)^{\frac{1}{1+\a}}) dz \\
&+ (T-t)^{\frac{2\a}{1+\a}} \int_{|z|>\r_0} K_{ij}( y (T-t)^{\frac{1}{1+\a}} - z) (v_i v_j)(z,t) dz \\
& = -\frac{1}{N}|v(y)|^2 + \int_{|z| \leq \r_0 / (T-t)^{\frac{1}{1+\a}}} K_{ij}( y  - z) (v_iv_j)( z ) dz + \tilde{p}(y,t).
\end{split}
\]
For $\tilde{p}$ we have a trivial pointwise estimate using the separation of $y (T-t)^{\frac{1}{1+\a}} $ and  $z$ inside the kernel:
\[
\tilde{p}(y,t) \lesssim (T-t)^{\frac{2\a}{1+\a}} \|u\|_2^2.
\]
Using \eqref{qdp} we continue the line above (suppressing the integrands for short):
\[
\begin{split}
&-\frac{1}{N}|v(y)|^2 + \int_{|z| \leq \r_0 / (T-t)^{\frac{1}{1+\a}}}  + \tilde{p}(y,t) \\ & = q(y) - d(t) = I(y)+h(y)-d(t) \\
&= -\frac{1}{N}|v(y)|^2+\int_{|z| \leq \r_0 / (T-t)^{\frac{1}{1+\a}}}+\int_{|z| > \r_0 / (T-t)^{\frac{1}{1+\a}}}
\\
&+h(y) - d(t).
\end{split}
\]
Denoting $\tilde{\tilde{p}}(y,t) = \int_{|z| > \r_0 / (T-t)^{\frac{1}{1+\a}}}$ we estimate as before,
\[
|\tilde{\tilde{p}}(y,t) | \lesssim  (T-t)^{\frac{2\a}{1+\a}}.
\]
So, from the identity above,
\[
h(y) - d(t) = \tilde{p}(y,t)-\tilde{\tilde{p}}(y,t),
\]
and thus $|h(y) - d(t)| \leq C (T-t)^{\frac{2\a}{1+\a}}$, for all $|y| \leq \frac{\r_0}{2(T-t)^{\frac{1}{1+\a}}}$. Given that $h$ is a polynomial, this can only be true if $h$ is constant and $d(t) = h + O((T-t)^{\frac{2\a}{1+\a}})$.
\end{proof}

\section{Proof of \thm{t:main}}
Let us note again that the upper bound in \eqref{e:2-bound} is a consequence of the energy conservation of the ambient solutions $u$. So, we focus on establishing the lower bound. Our starting point is the following local energy inequality derived in \cite{chae-shv}:
\[
\begin{split}
& \left|\frac{1}{l_2^{N-2\alpha}}\int_{|y|\leq
l_2}|v(y)|^2\sigma(y/l_2)dy-\frac{1}{l_1^{N-2\alpha}}\int_{|y|\leq
l_1}|v(y)|^2\sigma(y/l_1)dy\right| \\
& \leq C\int_{l_1/2\leq|y|\leq
l_2}\frac{|v|^3+|q||v|}{|y|^{N+1-2\alpha}}dy.
\end{split}
\]
Here $\s$ is a smooth cut-off function, $\s(y) = 1$ on $|y|<1$ and $\s = 0$ for $|y|>2$, and $l_1<l_2$. Assuming, on the contrary, that there is a sequence of $L_n$'s so that $L_n^{2\a} \ave{|v|^2}{L_n} \ra 0$ we let $l_2 = L_n$ and as a result in the limit obtain the following inequality (replacing $l_1/2$ with $L$)
\[
 \ave{|v|^2}{L} \lesssim \frac{1}{L^{2\a}}\int_{|y|>L} \frac{|v|^3+|q||v|}{|y|^{N+1-2\alpha}}dy.
\]
Rewriting it all in terms of averages and using H\"older on the pressure term we obtain
\[
\ave{|v|^2}{L} \lesssim \frac{1}{L} \sum_{k=1}^\infty \frac{1}{2^{k(1-2\a)}} \left( \aave{|v|^3}{2^kL}{2^{k+1}L} + \aave{|v|^3}{2^kL}{2^{k+1}L}^{1/3} \aave{|q|^{3/2}}{2^kL}{2^{k+1}L}^{2/3} \right),
\]
and as a consequence of \eqref{press-bound},
\begin{equation}\label{e:enin}
\ave{|v|^2}{L} \lesssim \frac{1}{L} \sum_{k=1}^\infty  \frac{1}{2^{k(1-2\a)}} \left( \aave{|v|^3}{2^kL}{2^{k+1}L}+ \aave{|v|^3}{2^kL}{2^{k+1}L}^{1/3} \sum_{l=1}^\infty \ave{|v|^2}{2^{k+l}L} \right)
\end{equation}

Now we initiate a bootstrap procedure on decay rates of the $L^3$ and $L^2$-averages. We will repeatedly use interpolation inequality with $\th = \frac{p-3}{p-2}$:
\[
\langle |v|^3 \rangle \leq \langle |v|^2 \rangle^{\th} \langle |v|^p \rangle^{1-\th}.
\]
So, from the start we have 
\[
\ave{|v|^2}{L} \lesssim \frac{1}{L^{2\a}}, \text{ and } \aave{|v|^{p}}{L}{2L} \lesssim \frac{1}{L^{N-\g}}.
\]
Let us denote $a_0 = 2\a$ and $c = N-\g$. By assumption, $a_0>0$ and $c>N+2-p$. By interpolation we then have
\[
\aave{|v|^{3}}{L}{2L} \lesssim \frac{1}{L^{\th a_0 + (1-\th)c}}.
\]
Denote $b_0 = \th a_0 + (1-\th)c$. Plugging this into \eqref{e:enin} we obtain a new decay rate for the energy, $\ave{|v|^2}{L} \lesssim \frac{1}{L^{a_1}}$, where $a_1$ is determined by the following condition: if $\frac32 a_0 \geq b_0$, then the energy terms on the r.h.s. of \eqref{e:enin} are of lower order, and thus $a_1 = b_0+1$; otherwise, $a_1 = \frac13 b_0+a_0+1$. Once $a_1$ is determined, the rate of $L^3$-average is obtained by interpolation again, $b_1 = \th a_1 + (1-\th) c$. Continuing this way we obtain a sequence of pairs $(a_0,b_0),(a_1,b_1),\ldots$ constructed by the same principle: $a_{n+1} = b_n+1$ provided $\frac32 a_n \geq b_n$, or otherwise, $a_{n+1} = \frac13 b_n + a_n+1$; then $b_{n+1} = \th a_{n+1} + (1-\th) c$. We now state the following claim.
\begin{claim}
If $0\leq \th <1$, $a_0\geq 0$, and $c > \frac{2\th - 3}{1-\th}$, then either $a_n \ra +\infty$ or $\lim_{n\ra \infty} a_n = c + \frac{1}{1-\th}$.
\end{claim}
Let us take the claim for granted for a moment. The assumptions of the claim are easy to verify given the hypotheses on $\g$ and $p$. In our terms we have $c+ \frac{1}{1-\th} = N-\g + p-2$. Since $\g < p-2$, we eventually reach the bound $\int_{|y|<L}|v|^2 dy \lesssim L^{-\e}$ for some $\e>0$, which implies $v \equiv 0$. It remains to prove the claim.

Suppose that on some $m$th step the inequality $\frac32 a_{m}\geq b_{m}$  is verified. Then $a_{m+1} = b_m+1$ and $b_{m+1} = \th a_{m+1} + (1-\th)c = \th b_m + \th+(1-\th)c$. Hence,
\[
\begin{split}
\frac 32 a_{m+1} & = \frac32 (b_m+1) = \frac32(\th a_m+(1-\th)c + 1)= \th \frac32 a_m + \frac32(1-\th)c + \frac32 \\
&> \th b_m+(1-\th)c + \th = b_{m+1},
\end{split}
\]
where the latter holds in view of the assumption on $c$. So, for subsequent pairs the algorithm stabilizes into a pattern. By recursion we obtain for $n>m$
\[
a_{n} = \th^{n-m} a_m+ ((1-\th)c+1)(1+\th+\cdots \th^{n-m-1}),
\]
which tends to $c+\frac{1}{1-\th}$ as $n\ra \infty$.  Now, if $\frac32 a_n< b_n$ holds for all $n$, we obtain by recursion 
\[
a_n = \g_1^n a_0 + (1+\g_1+\cdots + \g_1^{n-1})(\g_2 c+1),
\]
where $\g_1 = \frac13 \th+1$ and $\g_2 = \frac13 (1-\th)$. Noting that $\g_2 c+1>0$, we obtain $a_n \ra \infty$. This completes the proof of the theorem.

\medskip

It is curious to note that under the hypothesis of sublinear growth, $|v(y)| \lesssim |y|^{1-\d}$, one can alternatively prove the theorem via a more straightforward scheme without reliance on the higher $L^p$-bounds. Indeed, observe that
\[
\aave{|v|^3}{L}{2L} \lesssim L^{1-\d} L^{-2\a} = L^{-2\a +1 - \d}.
\]
The energy terms on the right hand side of \eqref{e:enin} are of smaller order on this and all the subsequent steps. Substituting into \eqref{e:enin} we find an improved bound $\ave{|v|^2}{L} \lesssim L^{-2\a - \d}$, and hence, $\aave{|v|^3}{L}{2L} \lesssim L^{-2\a - 2\d+1}$. Applying \eqref{e:enin} again, $\ave{|v|^2}{L} \lesssim L^{-2\a - 2\d}$.  On the $n$-th step we obtain $\ave{|v|^2}{L} \lesssim L^{-2\a - n\d}$, which clearly leads to $v \equiv 0 $.


\end{document}